\documentclass[a4paper,12pt]{article}
\usepackage{geometry,latexsym,amssymb,amsmath,amsthm,color}
\usepackage[latin5]{inputenc}
\usepackage{enumerate}
\usepackage{enumitem}
\usepackage[T1]{fontenc}
\usepackage{authblk}
\usepackage[all]{xy}
\usepackage{palatino}
\usepackage{indentfirst}
\usepackage{setspace}
\newtheorem{theorem}{Theorem}[section]
\newtheorem*{theorem A}{Theorem A}
\newtheorem*{theorem B}{N\"olker's Theorem}
\newtheorem{lemma}{Lemma}[section]
\newtheorem{proposition}{Proposition}[section]
\newtheorem{corollary}{Corollary}[section]

\theoremstyle{remark}
\newtheorem{remark}{Remark}[section]
\theoremstyle{remark}

\theoremstyle{definition}
\newtheorem{definition}{Definition}[section]
\newtheorem{example}{Example}[section]

\def\w{\widetilde}
\def\wtilde{\widetilde}

\def\St{\mathsf{St}}
\def\Cost{\mathsf{Cost}}
\def\Ob{\operatorname{Ob}}
\def\GpGd{\mathsf{GrpGpd}}
\def\XMod{\mathsf{XMod}}

\def\Aut{\operatorname{Aut}}
\def\inc{\operatorname{inc}}

\def\Ker{\operatorname{Ker}}
\def\Cok{\operatorname{Coker}}
\def\Im{\operatorname{Im}}
\def\NSGGd/G{\mathsf{NSGGd/G}}
\def\NSCM/(A,B,\alpha){\mathsf{NSCM/(A,B,\alpha)}}
\def\SpPXM/(A,B,\alpha){\mathsf{SpPXM/(A,B,\alpha)}}
\def\LXM/(A,B,\alpha){\mathsf{LXMod/(A,B,\alpha)}}
\def\LXM{\mathsf{LXMod}}
\def\CXM/(A,B,\alpha){\mathsf{CovXMod/(A,B,\alpha)}}
\def\GGdA(G){\mathsf{GpGpdAct(G)}}
\def\GGdC/G{\mathsf{GpGpdCov/G}}
\def\GGdCov/X{\mathsf{GpGpdCov/\pi X}}
\def\GdC/G{\mathsf{GpdCov/G}}
\def\TGrC/X{\mathsf{TGrCov/X}}
\def\GdA(G){\mathsf{GpdAct(G)}}
\def\St{\mathsf{St}}
\def\Act(G){\mathsf{GpdAct(G)}}
\def\Cov/G{\mathsf{GpdCov/G}}
\def\C{\mathsf{C}}

\def\epsilon{\varepsilon}

\begin{document}
\title{Group-groupoid actions and liftings of crossed modules}

\author[a]{Osman Mucuk\thanks{E-mail : mucuk@erciyes.edu.tr}}
\author[b]{Tunçar Şahan\thanks{E-mail : tuncarsahan@aksaray.edu.tr}}
\affil[a]{\small{Department of Mathematics, Erciyes University, Kayseri, TURKEY}}
\affil[b]{\small{Department of Mathematics, Aksaray University, Aksaray, TURKEY}}

\date{}

\maketitle
\begin{abstract}
The aim of this paper is  to define the notion of lifting of a crossed module via a group morphism and give some properties of this type of  the lifting.  Further we obtain a criterion for a crossed module to have a lifting of crossed module. We also prove that the liftings of   a certain crossed module constitute a category; and  that this category is equivalent   to the category of covers of that crossed module and hence  to the category of group-groupoid actions of the corresponding groupoid to that crossed module.
\end{abstract}

\\ {\bf Classification:} 20L05,  57M10,  22AXX, 22A22, 18D35
	
	\section{Introduction}
	The theory of covering groupoids has an important role in the applications of groupoids (cf.  \cite{Br1} and \cite{Hi}). In this theory it is well known that for a groupoid $G$, the category $\Act(G)$ of  groupoid actions  of $G$ on  sets, these are also called operations or $G$-sets, are equivalent to  the category $\Cov/G$ of  covering groupoids of $G$ (see \cite[Theorem 2]{Br-Da-Ha} for the topological version of this equivalence).
	
	On the one hand in  \cite[Proposition 3.1]{Br-Mu1} it was proved that if  $G$ is a group-groupoid, which is an internal groupoid in the category of groups and widely used in literature under the names \emph{2-group} (see for example \cite{baez-lauda-2-groups}), {\em $\mathcal{G}$-groupoid} or \emph{group object} \cite{BS1} in the category of groupoids, then   the category $\GGdC/G$ of group-groupoid coverings of $G$ is equivalent to the category $\GGdA(G)$ of group-groupoid actions of $G$ on groups. In \cite[Theorem 4.2]{Ak-Al-Mu-Sa} this result has been recently generalized to the case where $G$ is an internal groupoid for an algebraic category $\C$ which includes groups, rings without identity, $R$-modules,  Lie algebras, Jordan algebras, and many others; acting on a group with operations in the sense of Orzech \cite{Orz}.

	On the other hand in \cite[Theorem 1]{BS1} it was proved that the categories of crossed modules and group-groupoids, under the name of $\mathcal{G}$-groupoids, are equivalent (see also  \cite{Loday82} for an alternative equivalence in terms of an algebraic object called   {\em cat$^n$-groups}).  By applying this equivalence of the categories,  normal and quotient objects in the category of group-groupoids  have been recently obtained in \cite{Mu-Sa-Al}. In \cite[Section 3]{Por} it was proved that a similar result to the one in  \cite[Theorem 1]{BS1}  can be generalized  for a certain algebraic category introduced by Orzech \cite{Orz}, and adapted and called  category of groups with operations.  The study of internal category theory was continued in the works of Datuashvili \cite{Kanex} and \cite{Wh}. Moreover, she developed cohomology theory of internal categories, equivalently, crossed modules, in categories of groups with operations \cite{Dat} and \cite{Coh}. The  equivalences of the categories  in \cite[Theorem 1]{BS1} and  \cite[Section 3]{Por} enable  us to generalize some results on group-groupoids  to the more general internal groupoids for a certain  algebraic category $\C$ (see for example \cite{Ak-Al-Mu-Sa}, \cite{Mu-Be-Tu-Na},  \cite{Mu-Tu}  and \cite{Mu-Ak}).

	In this paper  we use the equivalence of group-groupoids and crossed modules proved in \cite[Theorem 1]{BS1} to determine  a notion in crossed modules,  called lifting of a crossed module, corresponding to a group-groupoid action, investigate some properties of lifting of crossed modules and give a criteria for a crossed module to have a lifting. Finally we prove that for a group-groupoid $G$, the group-groupoid actions of $G$ on groups and liftings of the crossed module corresponding to $G$ are categorically equivalent.	
	
	\section{Preliminaries}

	Let  $G$  be a groupoid.  We write  $\Ob(G)$  for the set of
	objects  of    $G $ and write $G$ for the set of morphisms. We also identify  $\Ob(G)$  with the set of
	identities of  $G $ and so an  element  of  $\Ob(G)$  may be written as
	$x$  or  $1_x$  as convenient.  We  write   $d_0, d_1 \colon
	G\rightarrow \Ob(G)$  for the source and target maps, and, as usual,
	write $G(x,y)$ for $d_0^{-1}(x)\cap d_1 ^{-1}(y)$, for $x,y\in \Ob(G)$.
	The composition  $h\circ g$  of two elements of  $G$  is defined if
	and only if  $d_0(h) =d_1(g)$, and so the   map $(h,g)\mapsto h\circ g$
	is defined on the pullback  $G {_{d_0}\times_{d_1}} G$
	of $d_0$  and $d_1 $.  The \emph{inverse} of $g\in G(x,y)$ is denoted by
	$g^{-1}\in G(y,x)$.
	
	If  $x\in \Ob(G) $, we write $\St_Gx$  for $d_0^{-1}(x) $  and  call
	the \emph{star} of $G$ at $x$. Similarly  we write $\Cost_Gx$ for
	$d_1^{-1}(x)$ and call \emph{costar} of $G$ at $x$. The set of all
	morphisms from $x$ to $x$ is a group, called \emph{object group}
	at $x$, and denoted by $G(x)$.
	
	A groupoid $G$ is \emph{transitive (resp. simply transitive, 1-transitive and totally intransitive)} if
	$G(x,y)\neq\emptyset$ (resp. $G(x,y)$ has no more than one element, $G(x,y)$ has exactly one element and
	$G(x,y)=\emptyset$) for all $x,y\in \Ob(G)$ such that $x\neq y$.
	
	A totally intransitive groupoid is determined entirely by the family
	$\{G(x)\mid  x\in \Ob(G)\}$ of  groups. This totally intransitive
	groupoid is sometimes called \emph{totally disconnected} or
	\emph{bundle of groups} \cite[pp.218]{Br1}.

	Let $p\colon\widetilde G\rightarrow G$ be a morphism of groupoids. Then $p$ is
	called a \emph{covering morphism} and $\widetilde{G}$  a \emph{covering groupoid} of $G$ if for
	each $\widetilde x\in \Ob(\widetilde G)$ the restriction  $\St_{\widetilde{G}}{\widetilde{x}} \rightarrow \St_{G}{p(\widetilde x)}$  is  bijective.
	
	Assume that $p\colon \wtilde{G}\rightarrow G$ is a covering morphism. Then  we have a lifting function $S_{p}\colon G_{d_{0}}\times_{\Ob(p)}\Ob(\wtilde{G})\rightarrow \wtilde{G}$ assigning to the pair $(a,x)$ in the pullback $G_{d_{0}}\times_{\Ob(p)}\Ob(\wtilde{G})$ the unique element $b$ of $\St_{\widetilde{G}}{x}$ such that $p(b)=a$. Clearly $S_{p}$ is inverse to $(p,d_{0})\colon  \widetilde{G}\rightarrow G_{s}\times_{\Ob(p)}\Ob(\widetilde{G})$. So it is stated that $p\colon  \widetilde{G}\rightarrow G$ is a covering morphism if and only if $(p,s)$ is a  bijection.

	A covering morphism $p\colon \widetilde{G}\rightarrow
	G$ is called \emph{transitive } if both $\widetilde{G}$ and $G$ are
	transitive.  A transitive covering morphism $p\colon\widetilde G\rightarrow G$
	is called \emph{universal} if $\widetilde G$ covers every cover of $G$, i.e.,  for every covering morphism $q\colon \widetilde{H}\rightarrow G$ there is a unique
	morphism of groupoids $\widetilde{p}\colon \widetilde G\rightarrow \widetilde{H}$ such that $q\widetilde{p}=p$
	(and hence $\widetilde{p}$ is also a covering morphism), this is equivalent to that for
	$\widetilde{x}, \widetilde{y}\in \Ob({\widetilde G})$ the set $\widetilde{G}(\widetilde x, \widetilde y)$
	has not more than one element.
	
	Recall that an action of a groupoid $G$ on a set $S$ via a function $\omega\colon S\rightarrow \Ob(G)$ is a function ${G}_{d_{0}}\times_\omega S\rightarrow S, (g,s)\mapsto g\bullet s$ satisfying the usual rules for an action: $\omega(g\bullet s)=d_1(g)$, $1_{\omega(s)}\bullet s=s$ and $(h\circ g)\bullet s=h\bullet (g\bullet s)$ whenever $h\circ g$ and $g\bullet s$ are defined. A morphism $f\colon (S,\omega)\rightarrow (S',\omega')$ of such actions is a function $f\colon S\rightarrow S'$ such that $w'f=w$ and $f(g\bullet s)=g\bullet f(s)$ whenever $g\bullet s$ is defined. This gives a category $\Act(G)$ of actions of $G$ on sets. For such an action the action groupoid $G\ltimes S$ is defined to have object set $S$, morphisms the pairs $(g,s)$ such that $d_0(g)=\omega(s)$, source and target maps $d_0(g,s)=s$, $d_1(g,s)=g\bullet s$, and the composition
	\[(g',s')\circ (g,s)=(g\circ g',s)\]
	whenever $s'=g\bullet s$. The projection $q\colon G\ltimes S\rightarrow G, (g,s)\mapsto s$ is a covering morphism of groupoids and the functor assigning this covering morphism to an action gives an equivalence of the categories $\Act(G)$ and $\Cov/G$.

	\section{Group-groupoids and crossed modules}
	
	A {\em group-groupoid} is a groupoid $G$ with  morphisms of groupoids $G\times G\rightarrow G$, $(g,h)\mapsto g+h$ and $G\rightarrow G,g\rightarrow -g$ yielding a group structure internal to the category of groupoids. Since the addition map is a morphism of groupoids, we have an interchange rule that $(b\circ a)+(d\circ c)=(b+d)\circ (a+c)$ for all $a,b,c,d\in G$  such that $b\circ a$ and $d\circ c$ are defined. If the identity of $\Ob(G)$ is $e$, then the identity of the group structure on the morphisms set  is $1_e$.

	Let $G$ be a group-groupoid. An action of the group-groupoid $G$ on a group $X$ via $\omega$ consists of a morphism $\omega\colon X\rightarrow \Ob(G)$ from the group $X$ to the underlying group of $\Ob(G)$ and an action of the groupoid $G$ on the underlying set $X$  via $\omega$ such that the following interchange law holds:
	\begin{equation}
		(g\bullet x)+(g'\bullet x')=(g+ g')\bullet(x+x') \label{interchangeaction}
	\end{equation}
	whenever both sides are defined. A morphism $f\colon (X,\omega)\rightarrow (X',\omega')$ of such actions is a morphism   $f\colon X\rightarrow X'$ of groups and of the underlying operations of $G$.  This gives a category $\GGdA(G)$ of actions of $G$ on groups. For an action of $G$ on the group $X$ via $\omega$, the action groupoid $G\ltimes X$ has a group structure defined by \[(g,x)+(g',x')=(g+g',x+x')\]
	and with this operation $G\ltimes X$ becomes a group-groupoid and the projection $p\colon G\ltimes X\rightarrow G$ is an object of the category $\GGdC/G$. By means of this construction the following equivalence of the categories was given in \cite[Proposition 3.1]{Br-Mu1}.
	
	\begin{proposition} The categories $\GGdC/G$ and $\GGdA(G)$ are equivalent.\end{proposition}

	We recall that as defined by Whitehead in \cite{Wth1,Wth2} a crossed module of groups consists of two groups $A$ and $B$,
	an action of $B$ on $A$ denoted by $b\cdot a$ for $a\in A$ and $b\in B$;
	and a morphism  $\alpha\colon A\rightarrow B$ of groups satisfying the
	following conditions for all $a,a_1\in A$ and $b\in B$
	\begin{enumerate}[label=\textbf{CM\arabic{*}}, leftmargin=2cm]
		\item\label{CM1} $\alpha(b\cdot a)=b+\alpha(a)-b$,
		\item\label{CM2} $\alpha(a)\cdot a_1=a+a_1-a$.
	\end{enumerate}
	We will denote such a crossed module by $(A,B,\alpha)$.
	
	Here are some examples of well known crossed modules:
	\begin{enumerate}[label=(\roman{*}), leftmargin=2cm]
		\item The inclusion map $N\hookrightarrow G$ of a normal subgroup
		is a crossed module with the conjugation action of $G$ on $N$.
		
		\item If $M$ is a $G$-module, then the zero
		map $0\colon M\rightarrow G$ has the structure of crossed module.
		
		\item The inner automorphism map
		$G\rightarrow Aut(G)$ is a crossed module.
		
		\item If $X$ is a topological group, then the fundamental group $\pi X$ is a group-groupoid, the star  $\St_{\pi X}0$ at the identity $0\in X$ becomes a group and the final point map $d_1\colon\St_{\pi X}0 \rightarrow X$ becomes a crossed module.
		
		\item As a motivating geometric example of crossed module due to Whitehead \cite{Wth1,Wth2} if $X$ is  topological space and $A\subseteq X$ with $x\in A$, then  there is a natural action of $\pi_1(A,x)$ on second relative homotopy group $\pi_2(X,A,x)$ and with this action the boundary map
		\[\partial\colon\pi_2(X,A,x)\rightarrow \pi_1(A,x)\] becomes a crossed module. This crossed module is called {\em fundamental crossed module}
		and denoted by $\Pi(X,A,x)$ (see for example \cite{BHS} for more details).
		
	\end{enumerate}

	The following are some standard properties of crossed modules.
	
	\begin{proposition}
		Let $(A,B,\alpha)$ be a crossed module. Then
		\begin{enumerate}[label=(\roman{*}), leftmargin=1cm]
			\item  $\alpha(A)$ is a normal subgroup of $B$.
			\item $\Ker \alpha$ is central in $A$, i.e. $\Ker \alpha$ is a subset of $Z(A)$, the center of $A$.
			\item $\alpha(A)$ acts trivially on $Z(A)$.
			\item $Z(A)$ and $\Ker \alpha$ inherit an action of $\Cok \alpha$ to become $(\Cok \alpha)$-modules.
		\end{enumerate}
	\end{proposition}

	Let $(A,B,\alpha)$ and $(A',B',\alpha')$ be two crossed modules. A morphism  $(f_1,f_2)$
	from $(A,B,\alpha)$ to $(A',B',\alpha')$  is a pair of morphisms of groups
	$f_1\colon A\rightarrow A'$ and $f_2\colon B\rightarrow B'$  such that
	$f_2\alpha=\alpha'f_1$ and  $f_1(b\cdot a)=f_2(b)\cdot f_1(a)$ for  $a\in A$ and $b\in B$.
	
	Crossed modules with morphisms between them form a category  denoted by $\XMod$.

	The following result was proved in \cite[Theorem 1]{BS1} and since we need some details of the proof, we give a sketch proof of the equivalence.
	
	\begin{theorem}
		\label{Theocatequivalence} The category $\XMod$  of crossed modules
		and the category $\GpGd$ of group-groupoids are equivalent.
	\end{theorem}
	\begin{proof}
		Let $\GpGd$ be the category of group-groupoids and  $\XMod$  the category
		of crossed modules.
		
		A functor $\delta\colon \GpGd\rightarrow \XMod$ is defined as follows. For a group-groupoid $G$ let  $\delta(G)$ be the crossed module $(A,B,d_1)$ where $A=\Ker d_0$, $B=\Ob(G)$ and $d_1\colon A\rightarrow B$ is the restriction of the target point map . Then $A, B$ inherit group structures from that of $G$, and the target point map
		$d_1\colon A\rightarrow B$ is a morphism of groups. Further we have an
		action $B\times A\rightarrow A$, $(b,a)\mapsto $ $b\cdot a$ of $B$ on
		the group $A$ given by
		\begin{equation}
			b\cdot a=1_{b}+a-1_{b} \label{actioneq}\end{equation}
		for  $a\in A$, $b\in
		B$ and we clearly have
		\begin{equation}
			d_1(b\cdot a)=b+d_1(a)-b \label{xmod1}
		\end{equation}
		and
		\begin{equation}
			d_1(a)\cdot a_{1}=a+a_{1}-a \label{xmod2}
		\end{equation}
		for $a,a_{1}\in A$, $b\in B$. Thus $(A,B,d_1)$ is a crossed module.
		
		Conversely define a functor $\eta\colon \XMod\rightarrow \GpGd$ in the following way. For a crossed module  $(A,B,\alpha)$  let $\eta(A,B,\alpha)$ be the group-groupoid whose object set (group)  is the
		group $B$ and whose  group of the morphisms is the semi-direct group
		$A\rtimes B$ with the usual group structure
		\begin{equation}
			(a_1,b_1)+(a,b)=( a_1+b_1\cdot a,b_1+b) \label{semidrectproduct}
		\end{equation}
		The  source and target  point maps are defined to be $d_0(a,b)= b$ and   $d_1(a,b)= \alpha(a)+b$, while the groupoid
		composition is given by \[(a_1,b_1)\circ (a,b)=(a_1+a,b)\] whenever
		$b_1=\alpha(a)+b$.
	\end{proof}

	\begin{proposition} Let $G$ be a  group-groupoid and $(A,B,\alpha)$ the  crossed module  corresponding to $G$. If $G$ is transitive (resp. simply transitive, 1-transitive and totally intransitive) then $\alpha$ is surjective (resp. injective,
		bijective; and zero morphism such that $A$ is abelian).
	\end{proposition}
	
	\begin{proof}
		The proof can be followed by the details of the proof of Theorem \ref{Theocatequivalence}.
	\end{proof}
	
	Hence we can state the following definition.
	\begin{definition}
		Let $(A,B,\alpha)$ be a crossed module. Then $(A,B,\alpha)$ is called \em{transitive (resp. simply transitive, 1-transitive and totally intransitive)} if $\alpha$ is surjective (resp. injective, bijective; and zero morphism such that  $A$ is abelian).\end{definition}
	
	\begin{example} If $X$ is a topological group whose underlying topology is path-connected (resp. totally disconnected), then the crossed module $(\St_{\pi X}0,X,d_1)$ is transitive (resp. totally intransitive).
	\end{example}

	\section{Liftings of crossed modules}
	
	In this section using Theorem \ref{Theocatequivalence}, we determine the notion in crossed modules, corresponding to the action of a group-groupoid on a group and interpret  the properties in crossed modules corresponding to group-groupoid actions  .
	
	Let $G$ be a group-groupoid acting on a group $X$ by an action $G_{d_0}\times_{\omega} X,(g,x)\mapsto g\bullet x$, via a group morphism $\omega\colon X\rightarrow \Ob(G)$ and let $(A,B,\alpha)$ be the crossed module corresponding to $G$. Then we have  a morphism $\omega\colon X\rightarrow B$ of groups and an action of $X$ on $A=\St_{G}0$ defined by
	\begin{equation}
		X\times A\rightarrow A, x\cdot a=1_{\omega(x)}+a-1_{\omega(x)} \label{Eqaction}
	\end{equation}
	By the group-groupoid action of $G$ on $X$ we have  a group morphism
	\begin{equation}
		\varphi\colon A\rightarrow X, a\mapsto \varphi(a)=a\bullet 0_X  \label{Eqgrpmorphism}
	\end{equation}
	such that $\omega\varphi=\alpha$, where $0_X$ is the identity element of the group $X$.
	
	Then we prove the following theorems.
	
	\begin{theorem} By the action of $X$ on $A$ defined above, $(A,X,\varphi)$ becomes a crossed module.
	\end{theorem}
	\begin{proof} We prove that the conditions [CM1] and [CM2] are satisfied.
		
		\begin{enumerate}[label=\textbf{[CM\arabic{*}]}, leftmargin=1.5cm]
			\item For  all $a\in A$ and $x\in X$ we have the following evaluations
			\begin{align*}
				\varphi(x\cdot a)  & = \varphi(1_{\omega(x)}+a-1_{\omega(x)}) \tag{by Eq. \ref{Eqaction}}\\
				& = (1_{\omega(x)}+a-1_{\omega(x)})\bullet 0_X\tag{by Eq. \ref{Eqgrpmorphism}}\\
				& = (1_{\omega(x)}+(a+1_{\omega(-x)}))\bullet (x+(-x)) \\
				& = (1_{\omega(x)}\bullet x)+((a+1_{\omega(-x)})\bullet(-x))\tag{by Eq. \ref{interchangeaction}}\\
				& = (0_A+1_{\omega(x)})\bullet(0_X+x)+(a+1_{\omega(-x)})\bullet(0_X+(-x))
			\end{align*}
			and by the interchange law ( \ref{interchangeaction})
			\begin{align*}
				\varphi(x\cdot a) & = (0_A\bullet 0_X)+(1_{\omega(x)}\bullet x)+(a\bullet 0_X)+(1_{\omega(-x)}\bullet(-x)) \\
				& = \varphi(0_A)+x+\varphi(a)-x \tag{by Eq. \ref{Eqgrpmorphism}}\\
				& = x+\varphi(a)-x
			\end{align*}
			\item For all $a,a_1\in A$ we have the following steps
			\begin{align*}
				\varphi(a)\cdot a_1  & = (a\bullet 0_X)\cdot a_1 \tag{by  Eq. \ref{Eqgrpmorphism}}\\
				& = 1_{\omega(a\bullet 0_X)}+a_1-1_{\omega(a\bullet 0_X)} \tag{by Eq. \ref{Eqaction}}\\
				& = 1_{d_1(a)}+a_1-1_{d_1(a)}   \tag{ by $\omega(a\bullet 0_X)=d_1(a)$}\\
				& = d_1(a)\cdot a_1 \tag{by Eq. \ref{actioneq}}\\
				& = a+a_1-a \tag{by Eq. \ref{xmod2}}
			\end{align*}
		\end{enumerate}
	\end{proof}
	\begin{theorem} \label{Acrosmodifredmorpis}
		Let $(A,B,\alpha)$ be a crossed module and $\omega\colon X\rightarrow B$ a group morphism. Then any group morphism
		$\varphi\colon A\rightarrow X$  such that  $\omega\varphi=\alpha$ is a  crossed module with the action  defined via $\omega$ if and only if  the map $ \overline{\varphi}\colon A\rtimes X\rightarrow X$ defined by
		\begin{equation}
			\overline{\varphi}(a,x)=\varphi(a)+x \label{reducedmorp}
		\end{equation}
		is a group morphism.
	\end{theorem}
	\begin{proof}  Suppose that the map  $ \overline{\varphi}\colon A\rtimes X\rightarrow X$  defined by $\overline{\varphi}(a,x)=\varphi(a)+x$ is a group morphism. Then we  have the following evaluations to prove that  the axioms  [CM1] and [CM2] of crossed module are satisfied for the group morphism $\varphi\colon A\rightarrow X$.
		\begin{enumerate}[label=\textbf{[CM\arabic{*}]}, leftmargin=2cm]
			\item  For $x\in X$ and $a\in A$ we have that
			\begin{align*}
				\varphi(x\cdot a) = \overline{\varphi}(x\cdot a,0) & = \overline{\varphi}((0,x)+(a,-x)) \tag{by  Eq. \ref{semidrectproduct}}\\
				& = \overline{\varphi}(0,x)+\overline{\varphi}(a,-x)\\
				& = \varphi(0)+x+\varphi(a)-x  \tag{by Eq. \ref{reducedmorp}}\\
				& = x+\varphi(a)-x
			\end{align*}
			
			\item  Since the action of $X$ on $A$ is defined by means of the action of $B$ on $A$ via $\omega$, the morphism  $\varphi$ satisfies the condition [\ref{CM2}] of crossed module.
			
		\end{enumerate}

		Conversely assume that $\varphi\colon A\rightarrow X$ is a  crossed module.  Then by the  following evaluation the map $ \overline{\varphi}\colon A\rtimes X\rightarrow X$ becomes a group morphism.
		\begin{align*}
			\overline{\varphi}((a,x)+(a_1,x_1))  & = \overline{\varphi}(a+x\cdot a_1,x+x_1) \tag{by  Eq. \ref{semidrectproduct}}\\
			& = \varphi(a+x\cdot a_1)+(x+x_1) \tag{by  Eq. \ref{reducedmorp}}  \\
			& = \varphi(a)+\varphi(x\cdot a_1)+(x+x_1) \\
			& = \varphi(a)+x+\varphi(a_1)-x+x+x_1 \tag{by  \ref{CM1}}\\
			& = \varphi(a)+x+\varphi(a_1)+x_1 \\
			& = \overline{\varphi}(a,x)+\overline{\varphi}(a_1,x_1) \tag{by  Eq. \ref{reducedmorp}}.
		\end{align*}
	\end{proof}

	We now define the notion of lifting of a crossed module as follows:
	
	\begin{definition}  Let $(A,B,\alpha)$ be a crossed module and $\omega\colon X\rightarrow B$ a morphism of groups.  Then a crossed module $(A,X,\varphi)$ in which  the action of $X$ on $A$ is defined via $\omega$, is a called a \emph{lifting of $\alpha$ over $\omega$} and  denoted by  $(\varphi,X,\omega)$ if the following diagram is commutative
		\[\xymatrix{ & X \ar[d]^\omega\\
			A \ar[r]_-\alpha \ar@{-->}[ur]^\varphi & B}\]
	\end{definition}
	
	It is obvious that every crossed module $(A,B,\alpha)$ lifts to itself over the identity morphism $1_B$ on $B$.
	
	\begin{remark}\label{remlift} In the following diagram if $(\varphi,X,\omega)$ is a lifting of $(A,B,\alpha)$, then $\Ker\varphi\subseteq\Ker\alpha$ and $(1_A,\omega)$ is a morphism of crossed modules
		\[\xymatrix{A \ar@{=}[d]_{1_A} \ar@{-->}[r]^\varphi & X \ar[d]^\omega\\
			A \ar[r]_-\alpha  & B}\]
		Therefore  if $(A,B,\alpha)$ is a simply transitive crossed module then $\Ker\alpha$ is trivial and so also $\Ker\varphi$ is. Hence the crossed module $(A,X,\varphi)$ is also simply transitive.
	\end{remark}
	
	The following can be stated as examples  of  lifting crossed modules.
	\begin{example}\label{autxmod}
		A crossed module $(A,B,\alpha)$ is a lifting of the automorphism crossed module $(A,\Aut(A),\iota)$ over the action of $B$ on $A$, i.e. $\theta\colon B\rightarrow \Aut(A)$, $\theta(b)(a)=b\cdot a$
		\[\xymatrix{  & B \ar[d]^\theta\\
			A \ar[r]_-\iota \ar[ur]^\alpha  & \Aut(A)}\]
	\end{example}
	
	\begin{example} Let $N$ be a normal subgroup of a group $G$. Since $(N,G,\inc)$ is a simply transitive crossed module, any lifting crossed module $(N,X,\varphi)$ is also simply transitive. Moreover $\varphi(N)$ is a normal subgroup of $X$. Since $\varphi$ is injective $N$ also can be consider as a normal subgroup of $X$.
	\end{example}
	
	\begin{example}
		Let $G$ be a group with trivial center. Then the automorphism crossed module $G\rightarrow \Aut{G}$ is simply transitive and hence every lifting of $G\rightarrow \Aut{G}$ is also simply transitive.
	\end{example}
	
	\begin{example}\label{topliftexam}
		Let $p\colon \w{X}\rightarrow X$ be a covering morphism of topological groups.  Hence  $(\St_{\pi X}0 ,X,d_{1})$ is a crossed module. If $\alpha$ is a path in $X$ with initial point $0\in X$, the identity, then by the path lifting property there exists a unique path $\w{\alpha}$ in $\w X$ such that $p\w{\alpha}=\alpha$ and $\w{\alpha}(0)=\tilde{0}\in \widetilde{X}$, the identity. Hence we can define a function $\w{d_1}\colon \St_{\pi X}0\rightarrow \w{X}$ assigning  $[\alpha]\in \St_{\pi X}0$ to the  final point $\w{\alpha}(1)$ of the lifting path of $\alpha$ at ${\w{0}}$. It follows that  $\w{d_1}$ is well defined and $(\w{d_1},\w{X},p)$ is a lifting of $(\St_{\pi X}0 ,X,d_1)$.
	\end{example}

	We now state some results on lifting crossed modules.
	\begin{lemma}
		Let $(A,B,\alpha)$ be a crossed module and $\varphi$ a lifting of $\alpha$ over $\omega\colon X\rightarrow B$. If
		there are isomorphisms $f\colon B\rightarrow B'$ and $g\colon X'\rightarrow X$ for some groups $B'$ and $G'$ then $\varphi'$ is a
		lifting of $\alpha'$ over $\omega'\colon X'\rightarrow B'$ where $\varphi'=g^{-1}\varphi$, $\alpha'=f\alpha$ and $\omega'=f\omega g$.
	\end{lemma}
	
	\begin{proof}
		It follows that $\alpha'=f\alpha$ and $\varphi'=g^{-1}\varphi$ are crossed modules since $f$ and $g$ are isomorphisms; and \begin{align*}
			\omega'\varphi'  &= (f\omega g)(g^{-1}\varphi) \\
			&=  f\omega \varphi
		\end{align*}
		Since $(\varphi,X,\omega)$ is a lifting of $(A,B,\alpha)$, i.e., $\omega \varphi=\alpha$ we have that
		\begin{align*}
			\omega'\varphi'  &= f\omega \varphi \\
			&=  f\alpha\\
			&=  \alpha'.
		\end{align*}
	\end{proof}
	
	\begin{proposition}\label{Proplifting}
		Let $(A,B,\alpha)$ be a crossed module and $(\varphi,X,\omega)$ a lifting of $(A,B,\alpha)$. If $(\varphi',X',\omega')$ is a lifting of $(A,X,\varphi)$ then $(\varphi',X',\omega\omega')$ is also a lifting of $(A,B,\alpha)$.
	\end{proposition}
	
	\begin{proof} The proof is immediate.\end{proof}

	Let $(\varphi,X,\omega)$ and $(\varphi',X',\omega')$ be two liftings of $(A,B,\alpha)$. A morphism $f$ from $(\varphi,X,\omega)$ to $(\varphi',X,\omega')$ is a group homomorphism $f\colon X\rightarrow X'$ such that $f\varphi=\varphi'$ and $\omega'f=\omega$. Hence lifting crossed modules of $(A,B,\alpha)$ and morphisms between them form a category which we denote by $\LXM/(A,B,\alpha)$. By  Proposition \ref{Proplifting} it follows that if $\varphi$ is a lifting of $(A,B,\alpha)$ over $\omega\colon X\rightarrow B$ then $\LXM/(A,X,\varphi)$ is a full subcategory of $\LXM/(A,B,\alpha)$.

	Let $(A,B,\alpha)$ be a transitive crossed module. Then a lifting  $(\varphi,X,\omega)$ of  $(A,B,\alpha)$   is called an \emph{$n$-lifting} when  $|\Ker \omega|=n$.
	\begin{corollary}
		If $(\varphi,X,\omega)$ is a 1-lifting of $(A,B,\alpha)$ then $\omega$ is an isomorphism. Hence $(A,X,\varphi)\cong(A,B,\alpha)$.
	\end{corollary}
	
	\begin{proof}
		If $(\varphi,X,\omega)$ is a 1-lifting of $(A,B,\alpha)$, then $\omega$ becomes surjective and $|\Ker \omega|=1$, i.e., $\omega$ is injective. Hence $\omega$ is an isomorphism.
	\end{proof}
	
	\begin{theorem}
		Let $(f,g)\colon(\widetilde{A},\widetilde{B},\widetilde{\alpha})\rightarrow(A,B,\alpha)$ be a morphism of  crossed modules  where $(\widetilde{A},\widetilde{B},\widetilde{\alpha})$ is transitive and let $(\varphi,X,\omega)$ be a lifting of $(A,B,\alpha)$. Then there is a unique morphism of crossed modules $(f,\widetilde{g})\colon(\widetilde{A},\widetilde{B},\widetilde{\alpha})\rightarrow(A,X,\varphi)$
		such that $\omega\widetilde{g}=g$ if and only if $f(\Ker \widetilde{\alpha})\subseteq\Ker \varphi$.
	\end{theorem}
	
	\begin{proof}
		Assume that $f(\Ker \widetilde{\alpha})\subseteq\Ker \varphi$. For the existence; let $\widetilde{b}\in \widetilde{B}$. Since $(\widetilde{A},\widetilde{B},\widetilde{\alpha})$ is transitive, $\widetilde{\alpha}$ is surjective and there exists an $\widetilde{a}\in \widetilde{A}$ such that $\widetilde{\alpha}(\widetilde{a})=\widetilde{b}$. Hence $\widetilde{g}(\widetilde{b})=\varphi f(\widetilde{a})$.
		
		It is easy to see that $\widetilde{g}$ is well defined since $f(\Ker \widetilde{\alpha})\subseteq\Ker \varphi$. Also
		\begin{align*}
			\omega\widetilde{g}(\widetilde{b})  &= \omega\varphi f(\widetilde{a}) \\
			&=  \alpha f(\widetilde{a}) \\
			&=  g\widetilde{\alpha}(\widetilde{a}) \\
			&=  g(\widetilde{b}).
		\end{align*}
		So $\omega\widetilde{g}=g$. By the definition of $\widetilde{g}$ it implies that $(f,\widetilde{g})$ is a crossed module morphism. For any other morphism $g'\colon \widetilde{B}\rightarrow X$ such that $\omega g'=g$, the pair $(f,g')$ is a crossed module morphism and coincides with $\widetilde{g}$.
		
		Conversely suppose that $(f,\widetilde{g})\colon(\widetilde{A},\widetilde{B},\widetilde{\alpha})\rightarrow(A,X,\varphi)$ is a crossed module morphism with $\omega\widetilde{g}=g$. If  $\widetilde{a}\in \Ker \widetilde{\alpha}$, then $f(\widetilde{a})\in f(\Ker \widetilde{\alpha})$ and $\widetilde{\alpha}(\widetilde{a})=0$. Since $(f,\widetilde{g})$ is a crossed module morphism $\widetilde{g}\widetilde{\alpha}=\varphi f$ and hence $\varphi f(\widetilde{a})=0$. So $f(\widetilde{a})\in \Ker \varphi$ and hence $f(\Ker \widetilde{\alpha})\subseteq\Ker \varphi$.
	\end{proof}
	
	\begin{corollary}\label{charsub}
		Let $(A,B,\alpha)$ be a crossed module. Assume that $(\varphi,X,\omega)$ and $(\widetilde{\varphi},\widetilde{X},\widetilde{\omega})$ are two liftings of $(A,B,\alpha)$ such that  $(A,\widetilde{X},\widetilde{\varphi})$ is transitive. Then  $\widetilde{\varphi}$ is a lifting of $\varphi$ if and only if $\Ker \widetilde{\varphi}\subseteq \Ker \varphi$.
	\end{corollary}
	
	\begin{corollary}
		Let $(A,B,\alpha)$ be a crossed module. Assume that $(\varphi,X,\omega)$ and $(\widetilde{\varphi},\widetilde{X},\widetilde{\omega})$ are two liftings of $(A,B,\alpha)$ such that  $(A,X,\varphi)$ and $(A,\widetilde{X},\widetilde{\varphi})$ are both transitive. Then  $(\varphi,X,\omega)\cong(\widetilde{\varphi},\widetilde{X},\widetilde{\omega})$ if and only if $\Ker \varphi=\Ker \widetilde{\varphi}$.
	\end{corollary}
	
	\begin{theorem}
		Let $(A,B,\alpha)$ be a crossed module, $X$  a group and let $\omega\colon X\rightarrow B$ be an injective group morphism.  Then any group morphism $\varphi\colon A\rightarrow X$ such that $\omega\varphi=\alpha$ becomes a lifting of $\alpha$ over $\omega$.
	\end{theorem}
	\begin{proof} According to Theorem \ref{Acrosmodifredmorpis} we only need to show that $\overline{\varphi}\colon A\ltimes X\rightarrow X$ defined by  $\overline{\varphi}(a,x)=\varphi(a)+x$  is a group morphism, i.e.  $\varphi(x\cdot  a)=x+\varphi(a)-x$ for all $x\in X$, $a\in A$ .
		\begin{align*}
			\omega(\varphi(x\cdot a)) & = \omega(\varphi(\omega(x)\cdot a)) \\
			& = \alpha(\omega(x)\cdot a) \tag{since $\omega\varphi=\alpha$}\\
			& = \omega(x)+\alpha(a)-\omega(x) \tag{by CM1}\\
		\end{align*}
		and
		\begin{align*}
			\omega(x+\varphi(a)-x) & = \omega(x)+\omega(\varphi(a))-\omega(x) \\
			& = \omega(x)+\alpha(a)-\omega(x) \tag{since $\omega\varphi=\alpha$.}\\
		\end{align*}
		Therefore since $\omega$ is injective and $\omega(\varphi(x\cdot a))=\omega(x+\varphi(a)-x)$ then $\varphi(x\cdot a)=x+\varphi(a)-x$ for all $x\in X$, $a\in A$. This completes the proof.
	\end{proof}
	
	\begin{corollary}\label{natlift}
		Every crossed module $(A,B,\alpha)$ lifts to the crossed module $(A,A/N,p)$ over $\omega\colon A/N\rightarrow B$, $a+N\mapsto \alpha(a)$ where $N=\Ker \alpha$.
		\[\xymatrix{ & A/N \ar[d]^\omega\\
			A \ar[r]_-\alpha \ar@{-->}[ur]^p & B}\]
	\end{corollary}
	
	This lifting is called \emph{natural lifting}. By Corollary \ref{natlift} and first isomorphism theorem  for groups we can say that for every crossed module $(A,B,\alpha)$, $(A,\Im \alpha,\alpha)$ is a transitive crossed module with the same action.
	
	Now we will give a criterion for the existence of the lifting crossed module.
	
	\begin{theorem}\label{existlift}
		Let $(A,B,\alpha)$ be a crossed module and $C$  a subgroup of $\Ker \alpha$. Then there exists a lifting $(\varphi,X,\omega)$ of $\alpha$ such that $\Ker \varphi=C$. Moreover in this case $\Ker\omega=\Ker \alpha / C$.
	\end{theorem}
	
	\begin{proof}
		Since $\Ker \alpha$ is central, $C$ is a normal subgroup of $A$. By Corollary \ref{natlift} $(\varphi,A/C,\omega)$ is a lifting of $\alpha$ where
		$\varphi(a)=a+C$ and $\omega(a+C)=\alpha(a)$. Obviously $\Ker \varphi=C$. Since $\varphi$ is surjective $(A,A/C,\varphi)$ is transitive. Further
		\begin{align*}
			\Ker\omega  & = \{a+C ~|~ \omega(a+C)=0\}\\
			& =  \{a+C ~|~ \alpha(a)=0\}\\
			& =  \{a+C ~|~ a\in\Ker\alpha\}\\
			& =  \Ker\alpha / C  \\
		\end{align*}
		and this completes the proof.
	\end{proof}
	
	\begin{definition} Let $(A,B,\alpha)$ be a crossed module. A lifting $(\varphi,X,\omega)$ of $(A,B,\alpha)$ is called \emph{transitive} if both crossed modules $(A,B,\alpha)$ and $(A,X,\varphi)$ are transitive.
	\end{definition}
	
	\begin{corollary}
		Let $(\varphi,X,\omega)$ be a transitive lifting of $(A,B,\alpha)$. Then $(\varphi,X,\omega)$ is an n-lifting if and only if $\left|\Ker\alpha/\Ker\varphi\right|=n$.
	\end{corollary}
	
	
	\begin{corollary}
		Let $X$ be a topological group and $C$ a subgroup of the fundamental group $\pi_1(X,0)$ of $X$ at the identity. Then there exists a lifting $(p,\St_{\pi X}0/C,\omega)$ of $(\St_{\pi X}0,X,d_1)$ such that $\Ker p=C$. Moreover if the underlying
		space of  $X$ is path-connected, then  $\St_{\pi X}0/\pi_1(X,0)\cong X$ as groups.
	\end{corollary}
	
	\begin{corollary}
		If  $X$ is a topological group whose underlying space is totally disconnected, then $\St_{\pi X}0 \cong \pi_1(X,0)$ as groups.
	\end{corollary}

	\begin{proposition}\label{1translift}
		If $(\varphi,X,\omega)$ is a 1-transitive lifting of a crossed module $(A,B,\alpha)$, then $(\varphi,X,\omega)$ is a lifting of any lifting of $(A,B,\alpha)$.
	\end{proposition}
	
	\begin{proof}
		Since $(A,X,\varphi)$ is 1-transitive, $\Ker\varphi$ is trivial and hence for any lifting $(\widetilde{\varphi},\widetilde{X},\widetilde{\omega})$ the kernel $\Ker\varphi$ is contained in $\Ker\widetilde{\varphi}$. Further by Corollary \ref{charsub} $(\varphi,X,\omega)$ is a lifting of $(\widetilde{\varphi},\widetilde{X},\widetilde{\omega})$.
	\end{proof}
	
	Hence one can  state the following definition.
	\begin{definition}
		A lifting of a crossed module $(A,B,\alpha)$ is called \emph{universal} if it  lifts to every lifting of $(A,B,\alpha)$.
	\end{definition}
	By Proposition \ref{1translift} a 1-transitive lifting is universal.
	
	As a consequence of Theorem \ref{existlift} we can give the following corollary.
	
	\begin{corollary}\label{unilift}
		Every crossed module has a universal lifting.
	\end{corollary}
	
	\begin{proof}
		Let $(A,B,\alpha)$ be a crossed module. If we choose $C$ as trivial in Theorem \ref{existlift}, then we have that $(1_A,A,\alpha)$ is a lifting of $(A,B,\alpha)$. Moreover $(1_A,A,\alpha)$ is a lifting of every lifting of $(A,B,\alpha)$.
	\end{proof}

	\section{Equivalences of the categories}
	
	In this section  we prove that for a certain group-groupoid $G$, there is a categorical  equivalence  between  the group-groupoid actions of $G$ on groups and lifting crossed modules of the  crossed module  corresponding to $G$.    We also prove that the liftings of a crossed module are  categorically equivalent to the covering morphisms of the same crossed module.

	\begin{theorem}
		Let $G$ be a group-groupoid and $(A,B,\alpha)$ the crossed module corresponding to $G$. Then the category $\GGdA(G)$ of group-groupoid actions of $G$ on groups and the category $\LXM/(A,B,\alpha)$ of lifting crossed modules of $(A,B,\alpha)$ are equivalent.
	\end{theorem}
	
	\begin{proof}
		Define a functor $\theta\colon\GGdA(G)\rightarrow\LXM/(A,B,\alpha)$ assigning each object    $(X,\omega)$  of  $\GGdA(G)$ to  a lifting  $(\varphi,X,\omega)$  of $(A,B,\alpha)$ where \[\varphi\colon A\rightarrow X, a\mapsto a\bullet 0_X\]
		
		Conversely define a functor $\psi\colon\LXM/(A,B,\alpha)\rightarrow\GGdA(G)$ assigning each lifting  $(\varphi,X,\omega)$ of $(A,B,\alpha)$  to a group-groupoid action $(X,\omega)$ of $G$ on the group $X$  via an action map defined by  \[{G}_{d_{0}}\times_\omega X\rightarrow X, (g,x)\mapsto g\bullet x=\varphi(g-1_{d_{0}(g)})+x.\]
		
		Now prove that $\theta\circ\psi$ and $\psi\circ\theta$ are respectively naturally isomorphic to the identity functors based on the categories $\LXM/(A,B,\alpha)$ and $\GGdA(G)$. If  $(\varphi,X,\omega)$ is a lifting of $(A,B,\alpha)$, then  $(\theta\circ\psi)(\varphi,X,\omega)=(\varphi',X,\omega)$ where $\varphi'$ is defined  by
		\[\varphi'(a) = a\bullet 0_X = \varphi(a-1_{d_{0}(a)})+0_X = \varphi(a) \]
		for all $a\in A$ and hence $\theta\circ\psi=1$.
		
		On the other hand if $(X,\omega)$ is an object of $\GGdA(G)$ with an action of $G$ on $X$ given by \[{G}_{d_{0}}\times_\omega X\rightarrow X, (g,x)\mapsto g\bullet x\] then the new action obtained from the functor $\psi$ is defined by
		\begin{align*}
			g\bullet' x  &= (g-1_{d_{0}(g)})\bullet (0_X + x) \\
			&=  (g-1_{d_{0}(g)})\bullet (0_X + (1_{\omega(x)})\bullet x)\\
			&=  (g-1_{d_{0}(g)}+ 1_{\omega(x)})\bullet (0_X+x) \tag{$1_{d_{0}(g)}=1_{\omega(x)}$}\\
			&=  g\bullet x.
		\end{align*}
		and therefore  $\psi\circ\theta=1$ which  completes the proof.
	\end{proof}

	Recall that as a result  of \cite[Proposition 4.2]{Br-Mu1}  a morphism $(f_1,f_2)\colon(\w{A},\w{B},\w{\alpha})\rightarrow (A,B,\alpha)$ of crossed modules such that $f_1\colon \w{A}\rightarrow A$ an isomorphism is called  {\em covering morphism}. Hence we can give the following theorem.
	\begin{theorem}\label{equivcovlift}
		For   a crossed module $(A,B,\alpha)$, the category $\LXM/(A,B,\alpha)$ of lifting crossed modules and the category $\CXM/(A,B,\alpha)$ of covering crossed modules of $(A,B,\alpha)$ are equivalent.
	\end{theorem}
	\begin{proof}
		Clearly if $(\varphi,X,\omega)$ is a lifting of $(A,B,\alpha)$, then $(1_A,\omega)\colon(A,X,\varphi)\rightarrow (A,B,\alpha)$ is a crossed module morphism. Hence $(A,X,\varphi)$ is a covering crossed module of $(A,B,\alpha)$. Moreover if $(\widetilde{\varphi},\widetilde{X},\widetilde{\omega})$ is another lifting of $(A,B,\alpha)$ and $f$ is a morphism in $\LXM/(A,B,\alpha)$ from $(\varphi,X,\omega)$ to $ (\widetilde{\varphi},\widetilde{X},\widetilde{\omega})$,  then
		$(1_A,f)$ is a morphism in $\CXM/(A,B,\alpha)$ from $(A,X,\varphi)$ to $ (A,\widetilde{X},\widetilde{\varphi})$ . This construction is clearly functorial.
		
		Conversely if $(f_2,f_1)\colon(\w{A},\w{B},\w{\alpha})\rightarrow (A,B,\alpha)$ is a covering morphism of crossed modules, then $f_2\colon \w{A}\rightarrow A$ is an isomorphism and $\varphi=\w{\alpha}f_2^{-1}$ is the lifting of $\alpha$ over $f_1$. It is easy to see that $f_1\varphi=\alpha$ and $(A,\widetilde{B},\varphi)$ is a crossed module. Further  if \[(g_2,G)\colon(A',B',\alpha')\rightarrow (A,B,\alpha)\] is another covering of $(A,B,\alpha)$ and $(h_2,h_1)\colon(\w{A},\w{B},\w{\alpha})\rightarrow (A',B',\alpha')$ is a morphism in $\CXM/(A,B,\alpha)$ then $h_1\colon(\varphi,\widetilde{B},f_1)\rightarrow (\varphi',B',G)$ is a morphism in $\LXM/(A,B,\alpha)$. This construction is also functorial and the other details of the equivalence is straightforward.
	\end{proof}

	By Example \ref{autxmod} and Theorem \ref{equivcovlift} we can state the following corollary.
	
	\begin{corollary}
		Every crossed module $(A,B,\alpha)$ is a covering of the automorphism crossed module $(A,\Aut(A),\iota)$ constructed by $A$.
	\end{corollary}
	
	The category $\CXM/(A,B,\alpha)$ of crossed module coverings of a certain crossed module $(A,B,\alpha)$ is equivalent to the category $\GGdC/G$ of  group-groupoid coverings of  group-groupoid  corresponding to $(A,B,\alpha)$ (see also \cite{Ak-Al-Mu-Sa,Mu-Tu} for more general case of  crossed modules over groups with operations for an algebraic category $\C$). Hence by  Theorem \ref{equivcovlift} we can obtain the following corollary.
	
	\begin{corollary}\label{equiv}
		Let $G$ be a group-groupoid and $(A,B,\alpha)$  the corresponding crossed module. Then the category $\GGdC/G$ of group-groupoid coverings of $G$ and the category $\LXM/(A,B,\alpha)$ of lifting crossed modules of $(A,B,\alpha)$ are equivalent.
	\end{corollary}
	
	By Corollary \ref{unilift} and Corollary \ref{equiv} we can obtain the following result.
	
	\begin{corollary}
		Every group-groupoid has a universal covering group-groupoid.
	\end{corollary}
	
	By Corollary \ref{natlift} and Corollary \ref{equiv} we can obtain the following result.
	
	\begin{corollary}
		Let $G$ be a group-groupoid. Then $G$ acts on $X=\Ker d_{0}/G(0)$ over $\omega\colon\Ker d_{0}/G(0)\rightarrow \Ob(G), \omega(a+G(0))=d_1(a)$ as a group-groupoid action with
		\[\varphi\colon {G}_{d_{0}}\times_\omega X\rightarrow X, \varphi(b,a+G(0))=(b\circ a)+G(0)\]
	\end{corollary}
	
	In \cite[Proposition 2.3]{Br-Mu1} it  was  proved that for a topological group $X$ whose underlying space has a universal cover,  the category $\TGrC/X$ of topological coverings of $X$ and the category $\GGdCov/X$ of coverings of group-groupoid $\pi X$  are equivalent; and in  \cite[Theorem 4.1]{Ak-Al-Mu-Sa}  a similar result was given  for more general topological  groups with operations. Hence we can state the  following corollary.
	
	\begin{corollary}
		Let $X$ be a topological group whose underlying space has a universal cover. Then the category $\TGrC/X$ of covers of $X$ in the category of topological groups and the category \[\LXM/(\St_{\pi X}0,X,d_1)\] of liftings of $(\St_{\pi X}0,X,d_1)$ are equivalent.
	\end{corollary}

	\section{Conclusion}
	
	Using the results of the paper \cite{Mu-Sa-Al} it  could be possible to develop normal and quotient notions of lifting crossed modules. Moreover using the equivalence of the categories in \cite[Section 3]{Por} as stated in the introduction, parallel  results of this paper could be obtained for crossed modules of group with operations and internal groupoids working in a certain algebraic category $\C$.


\begin{thebibliography}{99}
		
		\bibitem{Ak-Al-Mu-Sa} {H.F. Ak\i z, N. Alemdar, O. Mucuk, T. \c{S}ahan},
		{\em Coverings of internal groupoids and crossed modules in the category of  groups with operations},
		Georgian Math. Journal,  {\bf 20 (2)} (2013), 223--238.
		
		\bibitem{Br1} {R. Brown},
		\emph{Topology and Groupoids},
		BookSurge LLC, North Carolina, 2006.
		
		\bibitem{baez-lauda-2-groups}   {J.C. Baez, A.D. Lauda},
		\emph{Higher-dimensional algebra V: 2-groups}
		Theory Appl. Categ. \textbf{12} (2004), 423--491.
		
		\bibitem{BHS}   {R. Brown, P.J. Higgins, R. Sivera},
		\emph{Nonabelian Algebraic Topology: filtered spaces, crossed complexes, cubical homotopy groupoids}, European Mathematical Society Tracts in Mathematics 15, 2011.
		
		\bibitem{Br-Da-Ha}   {R. Brown, G. Danesh-Naruie, J.P.L. Hardy},
		\emph{Topological Groupoids: II. Covering Morphisms and G-Spaces},
		Math. Nachr. \textbf{74} (1976), 143--156.
		
		\bibitem{Br-Mu1}   {R. Brown, O.Mucuk},
		\emph{Covering groups of non-connected topological groups revisited},
		Math. Proc. Camb. Phil. Soc.  \textbf{115} (1994), 97--110.
		
		\bibitem{BS1}   {R. Brown, C.B. Spencer},
		\emph{G-groupoids, crossed modules and the fundamental groupoid of a topological group},
		Proc. Konn. Ned. Akad. v. Wet. \textbf{79} (1976), 296--302.
		
		\bibitem{Dat}   {T. Datuashvili},
		\emph{Cohomologically trivial internal categories in categories of groups with operations},
		Appl. Categ. Structures \textbf{3} (1995), 221--237.
		
		\bibitem{Coh}   {T. Datuashvili},
		\emph{Cohomology of internal categories in categories of groups with operations, Categorical Topology and its Relation to Analysis, algebra and combinatorics},
		Ed. J. Adamek and S. Mac Lane (Prague, 1988), World Sci. Publishing, Teaneck, NJ, 1989.
		
		\bibitem{Kanex}   {T. Datuashvili},
		\emph{Kan extensions of internal functors: Nonconnected case},
		J.Pure Appl.Algebra \textbf{167} (2002), 195--202.
		
		\bibitem{Wh}   {T. Datuashvili},
		\emph{Whitehead homotopy equivalence and internal category equivalence of crossed modules in categories of groups with operations},
		Proc. A. Razmadze Math.Inst. \textbf{113} (1995), 3--30.
		
		\bibitem{Hi}   {P.J. Higgins},
		\emph{Categories and groupoids},
		Van Nostrand, New York, 1971.
		
		\bibitem{Loday82}   {J.-L. Loday},
		\emph{Spaces with finitely many non-trivial homotopy groups},
		J. Pure Appl. Algebra \textbf{24} (1982), 179--202.
		
		\bibitem{Mu-Be-Tu-Na}   {O. Mucuk, B. K{\i}l{\i}\c{c}arslan, T. \c{S}ahan, N. Alemdar},
		\emph{Group-groupoids and monodromy groupoids}
		Topology Appl. \textbf{158 (15)} (2011), 2034--2042.
		
		\bibitem{Mu-Tu}   {O. Mucuk, T. \c{S}ahan},
		\emph{Coverings and crossed modules of topological groups with operations},
		Turk. J. Math. \textbf{38 (5)} (2014), 833--845.
		
		\bibitem{Mu-Sa-Al}   {O. Mucuk, T. \c{S}ahan, N. Alemdar},
		\emph{Normality and Quotients in Crossed Modules and Group-groupoids},
		Appl. Categor. Struct. \textbf{23} (2015), 415--428.
		
		\bibitem{Mu-Ak}   {O. Mucuk, H.F. Ak\i z},
		\emph{Monodromy groupoids of an internal groupoid in topological groups with operations},
		Filomat \textbf{29 (10)} (2015), 2355--2366.
		
		\bibitem{Orz}   {G. Orzech},
		\emph{Obstruction theory in algebraic categories I, II.}
		J. Pure Appl. Algebra \textbf{2} (1972), 287--314, 315--340.
		
		\bibitem{Por}   {T. Porter},
		\emph{Extensions, crossed modules and internal categories in categories of groups with operations},
		Proc. Edinb.  Math. Soc. \textbf{30} (1987), 373--381.
		
		\bibitem{Wth1}   {J.H.C. Whitehead},
		\emph{Note on a previous paper entitled ``On adding relations to homotopy group"}.
		Ann. of Math. \textbf{47} (1946), 806--810.
		
		\bibitem{Wth3}   {J.H.C. Whitehead},
		\emph{On operators in relative homotopy groups},
		Ann. of Math. \textbf{49} (1948), 610--640.
		
		\bibitem{Wth2}   {J.H.C. Whitehead},
		\emph{Combinatorial homotopy  II},
		Bull. Amer. Math. Soc. \textbf{55} (1949), 453--496.
	\end{thebibliography}
\end{document}